\documentclass[a4paper, 11pt]{article}

\usepackage{amsmath}
\usepackage{amsfonts}
\usepackage{amssymb}
\usepackage[english]{babel}
\usepackage{graphicx}
\usepackage{amsthm}
\usepackage{accents}

\def\Z{\mathbb{Z}}
\newtheorem{proposition}{Proposition}
\newtheorem{theorem}{Theorem}

\newtheorem{definition}{Definition}
\newtheorem{corollary}{Corollary}

\begin{document}

\title{Hadamard partitioned difference families\\and their descendants}

\author{Marco Buratti \thanks{Dipartimento di Matematica e Informatica, Universit\`a di Perugia, via Vanvitelli 1 - 06123 Italy, email: buratti@dmi.unipg.it}}

\maketitle
\begin{abstract}
\noindent
If $D$ is a $(4u^2,2u^2-u,u^2-u)$ Hadamard difference set (HDS) in $G$, then
$\{G,G\setminus D\}$ is clearly a $(4u^2,[2u^2-u,2u^2+u],2u^2)$ partitioned difference family (PDF).
Any $(v,K,\lambda)$-PDF will be said of Hadamard-type if $v=2\lambda$ as the one above. 
We present a doubling construction which, starting from any such PDF, leads to an infinite class of PDFs.
As a special consequence, we get a PDF in a group of order $4u^2(2n+1)$
and three block-sizes $4u^2-2u$, $4u^2$ and $4u^2+2u$, whenever we have a $(4u^2,2u^2-u,u^2-u)$-HDS 
and the maximal prime power divisors of $2n+1$ are all greater than $4u^2+2u$.
\end{abstract}


\normalsize
\section{Introduction}

Throughout this note the {\it multiset sum} of two multisets $X$ and $Y$ on a set $S$
is the multiset $X \ \uplus \ Y$ where the multiplicity of every element of $S$ is the sum of its multiplicities in $X$ and $Y$.
The multiset sum of $\mu$ copies of a multiset $X$ will be denoted by $^\mu X$. The multiset $^{\mu_1}\{x_1\} \ \uplus \ \dots \ \uplus \ ^{\mu_t}\{x_t\}$ will be
denoted by $[^{\mu_1}x_1,\dots,\,^{\mu_t}x_t]$. 

Given an additive group $G$ and a multiset $X=\{x_1,\dots,x_k\}$ on $G$, the {\it list of differences of $X$}
is the multiset $\Delta X$ of all possible differences $x_i-x_j$ with $(i,j)$ an ordered pair of distinct elements in $\{1,\dots,k\}$.
More generally, the list of differences of a collection $\cal X$ of multisubsets of $G$ is  the multiset sum
$\Delta{\cal X}:=\displaystyle\biguplus_{X\in{\cal X}}\Delta X$. Assume that $K$ is the multiset of sizes
of the members of $\cal X$. The following cases are important:
\begin{itemize}
\item  $\Delta{\cal X}= \ ^\lambda(G\setminus H)$ for a suitable $\lambda$
and a suitable subgroup $H$ of $G$. Here one says that $\cal X$ is a $(G,H,K,\lambda)$ 
{\it difference family} (DF).
\item $\Delta{\cal X}= \ ^\mu G$ for a suitable $\mu$. 
Here one says that $\cal X$ is a $(G,K,\mu)$ {\it strong difference family} (SDF).
\end{itemize}

Very often one refers to a $(G,H,K,\lambda)$-DF as a $(v,h,K,\lambda)$ difference family in $G$ {\it relative to $H$} 
where $v$ and $h$ denote the orders of $G$ and $H$, respectively.
One speaks of an {\it ordinary} difference family when $H=\{0\}$.
In this case one simply writes $(G,K,\lambda)$-DF or $(v,K,\lambda)$-DF in $G$.

The members of a DF or SDF are called {\it blocks}.
It is clear that every block of a DF is a set while a SDF must have at least one block with repeated elements.
A DF or SDF with only one block is said to be a {\it difference set} (DS) or a {\it difference multiset} (also called a 
{\it difference cover} in \cite{AS,M}), respectively. 
If one writes $k$ instead of $K$, it means that all blocks have size $k$.

A $(G,H,K,\lambda)$-DF is said to be {\it partitioned} (PDF) if its blocks partition $G\setminus H$.
Ordinary PDFs, also known as {\it zero difference balanced functions} (see, e.g., \cite{WZ, ZTWY}), 
have been introduced by Ding and Yin \cite{DY} for the construction
of {\it optimal constant composition codes}. They are also important from the
design theory perspective; for instance, it is shown in \cite{BYW} that a PDF with $K=[^1(k-1),\,^rk]$ 
gives rise to a {\it resolvable $2$-design}  with block size $k$. 

In this paper we are interested in $(v,K,\lambda)$-PDFs having $v=2\lambda$.
The motivation will be given by our main construction: 
each PDF with this property leads to an infinite class of new PDFs not obtainable 
with the classic composition constructions making use of difference matrices
\cite{BYW,LWG}. 

\begin{definition}
A Hadamard PDF is a $(G,K,\lambda)$-PDF with $|G|=2\lambda$. 
\end{definition}

We chose the name ``Hadamard" since every {\it Hadamard difference set} (HDS)
immediately gives a PDF with the required property. For convenience of the reader we recall that a HDS is a
difference set of parameters\break $(4u^2,2u^2-u,u^2-u)$ for some $u$ (see, e.g., \cite{JPS}).

\begin{proposition}\label{D,G-D}
If $D$ is a HDS in $G$, then $\{D,G\setminus D\}$ is a Hadamard PDF.
\end{proposition}
\begin{proof}
It is known and trivial that if $D$ is a $(v,k,\lambda)$-DS in $G$, then $\{D,G\setminus D\}$\break is a 
$(v,[k,v-k],v-2k+2\lambda)$-PDF in $G$. So, in particular, if $D$ is a\break $(4u^2,2u^2-u,u^2-u)$-HDS, 
then $\{D,G\setminus D\}$ is a $(4u^2,[2u^2-u,2u^2+u],2u^2)$-PDF in $G$. The assertion follows.
\end{proof}

At the moment to find another class of Hadamard PDFs seems to be quite hard to this author.
Anyway the set of Hadamard PDFs not coming from the above proposition is not empty. Indeed we have found
a $(32,[^22,6,22],16)$-PDF in the non-abelian group $G$ whose elements are all pairs of the Cartesian product 
$\Z_4\times\Z_8$ and whose operation law
is $$(x_1,y_1)+(x_2,y_2)=(x_1+x_2,5^{x_2}y_1+y_2).$$
One can check that the four blocks of the mentioned PDF are the following:
$$X_1=\{(0,0),(2,0)\};\quad X_2=\{(1,0),(3,4)\}$$
$$X_3=\{(0,1),(0,3),(1,2),(1,5),(1,6),(3,3)\};\quad X_4=G\setminus(X_1 \cup X_2 \cup X_3).$$

\section{Hadamard strong difference families}

Strong difference families have been formally introduced in \cite{Bsdf} but they have been implicitly used in many
earlier papers. Note that they have no known relation to the
``strong external difference families'' very recently introduced in
\cite{PS}, in spite of the misfortune of inadvertently  similar terminology. 

Strong difference families might be very useful to construct relative difference families \cite{BG,BP,CFW,M};
in a certain sense they are the ``skeleton" of the resulting difference family as shown in the following
construction which is a little bit more general than the ``fundamental construction" in \cite{Bsdf}.

\begin{theorem}\label{fund}
Let ${\cal X}=\{X_1,\dots,X_t\}$ be a $(G,K,\mu)$-SDF. Take a group $H$ and a pair of positive integers $(e,\lambda)$
with $\mu e=\lambda(|H|-1)$. Then take a collection ${\cal B}=\{B_1,\dots,B_t\}$ of subsets of $G\times H$ such that the projection of $B_i$ on $G$ coincides with $X_i$
for $i=1,\dots,t$. We have
$$\Delta{\cal B}=\biguplus_{g\in G}\{g\}\times L_g$$
for suitable lists $L_g$ of elements of $H$. Here, by definition of a SDF, 
every $L_g$ has constant size $\mu$. This is important because it allows for the possible existence of 
an $e$-set $E$ of endomorphisms of $H$ for which we have
$$\displaystyle\biguplus_{\varepsilon\in E}\varepsilon(L_g)= \ ^{\lambda}(H\setminus\{0\})\quad\forall g\in G.$$
If such a set $E$ exists, extend each $\varepsilon \in E$ to the endomorphism $\widehat\varepsilon$ of $G\times H$ defined 
by $\widehat\varepsilon(g,h)=(g,\varepsilon(h))$ for each $(g,h)\in G\times H$. 
Then  $\{\widehat\varepsilon(B_i) \ | \ 1\leq i \leq t; \varepsilon\in E\}$ is a $(G\times H, G\times\{0\},\ ^eK,\lambda)$-DF.
\end{theorem}

\bigskip
The ``fundamental construction" in \cite{Bsdf} corresponds to the case that $H$ is the additive group of a finite field $\mathbb{F}_q$
and each $\varepsilon\in E$ is the multiplication by a suitable non-zero element of the field. Thus, it is in some way related with the 
``factorization of a group" \cite{S} and, more generally, with the ``multifold factorization of a group" \cite{Bpairwise, Jimbo}.
Indeed the construction succeeds if there exists $E\subset\mathbb{F}_q^*$ such that  
$E\cdot L_g=\,^\lambda\mathbb{F}_q^*$ for every $g\in G$. In most known applications $E$ is a subgroup of $\mathbb{F}_q^*$ and each
$L_g$ is the multiset sum of $\lambda$ complete systems of representatives
for the cosets of $E$ in $\mathbb{F}_q^*$. 

In \cite{Bsdf} it was shown how ``playing" with some classic difference sets it is possible to obtain elementary but very useful
strong difference families. For instance, it was shown that if $D$ is a {\it Paley-type} difference set in a group $G$ of order $4n-1$ - hence of parameters
$(4n-1,2n-1,n-1)$ - then $^2(G\setminus D)$ is a $(4n-1,4n,4n)$ difference multiset. Here we will do something similar using Hadamard PDFs.

\begin{proposition}\label{sdf}
If ${\cal X}$ is a Hadamard $(G,K,\lambda)$-PDF,
then $\{^2X \ | \ X\in{\cal X}\}$ is a $(G,2K,4\lambda)$-SDF.
\end{proposition}
\begin{proof}
Let $m(g)$ be the multiplicity of $g$ in $\displaystyle\biguplus_{X\in{\cal X}}\Delta(\,^2 X)$. 
We have to prove that  $m(g)=4\lambda$ for all $g\in G$.
It is quite evident that if $X$ is a subset of $G$, then $\Delta(^ r X)$
contains zero exactly $r(r-1)|X|$ times and it contains a non-zero element $g$ of $G$ exactly 
$r^2\lambda_X(g)$ times where $\lambda_X(g)$ is the multiplicity of $g$ in $\Delta X$. 
Applying this to our case we get 
$$m(0)=2\displaystyle\sum_{X\in {\cal X}}|X|\quad\quad\mbox{and}\quad\quad
m(g)=4\displaystyle\sum_{X\in{\cal X}}\lambda_X(g)\quad  \forall g\in G\setminus\{0\}.$$
Considering that $\cal X$ is a $(G,K,\lambda)$-PDF, we have $\sum_{X\in {\cal X}}|X|=|G|$ and 
$\sum_{X\in{\cal X}}\lambda_X(g)=\lambda$. Also, considering that $\cal X$ is a Hadamard PDF, we have $|G|=2\lambda$. Hence we can write:
$$m(0)=2|G|=4\lambda\quad\quad\mbox{and}\quad\quad
m(g)=4\lambda\quad  \forall g\in G\setminus\{0\}.$$
The assertion follows.\end{proof}

We will refer to the strong difference family constructed in the above proposition as the Hadamard SDF associated with $\cal X$.

\section{The main construction}
We are now ready to state and prove our main result.
\begin{theorem}\label{HPDF}
Let ${\cal D}$ be a Hadamard $(G,K,\lambda)$-PDF and let $R=(H,+,\cdot)$ be a ring with identity of order $2n+1$
admitting a set $Y$ of units such that:
\begin{itemize}
\item[]$Y$ has size equal to the maximum size $K_{max}$ of the blocks of $\cal D$;
\item[]every element of $\Delta(Y \ \cup \ -Y)$ is a unit of $R$.
\end{itemize}
Then there exists both a
$$(2\lambda(2n+1),\ ^n(2K) \ \uplus \ K, \ 2\lambda)\mbox{-PDF}$$ and a
$$(2\lambda(2n+1),\ ^n(2K) \ \uplus \ \{2\lambda\}, \ 2\lambda)\mbox{-PDF}$$
 in $G\times H$.
\end{theorem}
\begin{proof} We get the result by applying Theorem \ref{fund} with ${\cal X}$ the Hadamard $(G,2K,4\lambda)$-SDF associated with $\cal D$. 
So, if ${\cal D}=\{D_1,\dots, D_t\}$, we have\break ${\cal X}=\{X_1,\dots,X_t\}$ with $X_i= \,^2D_i$ for $i=1,\dots, t$.

Take any map $f: G\longrightarrow Y$ which is injective on each block $D_i\in{\cal D}$. This is possible since the blocks
of $\cal D$ are disjoint by definition and we have
$|Y|=K_{\max}\geq|D_i|$ for every $i$.
Now, for $i=1,\dots,t$, consider the subset $B_i$ of $G\times H$ defined by
$$B_i=\biguplus_{d\in D_i}\{d\}\times\{f(d), -f(d)\}.$$
As prescribed by Theorem \ref{fund} the projection of $B_i$ on $G$ coincides with $X_i$ for $i=1, \dots, t$.
So we have $\Delta\{B_1, \dots, B_t\}=\biguplus_{g\in G}\{g\}\times L_g$ where each $L_g$ is a list of $4\lambda$ elements of $H$.
Explicitly, these lists are as follows:
$$L_0=\{\pm 2f(d) \ | \ d\in G\};$$
$$L_g=\{\pm f(d)\pm f(d') \ | \ (d,d')\in \biguplus_{i=1}^t D_i\times D_i; \ d-d'=g\}\quad\mbox{for $g\neq0$}$$
with all possible choices of the signs.
We notice two things: all elements of these lists are units in view of the properties of $Y$; each $L_g$ is closed under
taking opposites ($h$ and $-h$ have the same multiplicity in $L_g$ for every pair $(g,h)\in G\times H$).
Thus we can write $L_g=\{1,-1\}\cdot L'_g$ where each $L'_g$ is a list of $2\lambda$ units of $R$. 

For every $h\in H$, let us denote by $\varepsilon_h$ the endomorphism of $(H,+)$ which is the multiplication 
by $h$ in the ring $R$. Take a complete set $S$ of representatives for the pairs of the patterned starter\footnote{The {\it patterned starter} of an additive group $H$ of odd order is the set of all possible pairs
$\{h,-h\}$ of opposite elements of $H\setminus\{0\}$ (see, e.g., \cite{D}).}
of $(H,+)$ and consider the set $E=\{\varepsilon_s \ | \ s\in S\}$. We have:
$$\biguplus_{\varepsilon\in E}\varepsilon(L_g)=\biguplus_{s\in S}\{s,-s\}\cdot L'_g=(H\setminus\{0\})\cdot L'_g= \,^{2\lambda}(H\setminus\{0\})$$
the last equality being true since every element of $L'_g$ is a unit and $|L'_g|=2\lambda$. Keeping the same notation used in Theorem \ref{fund}
we conclude that 
$${\cal F}:=\{\widehat\varepsilon_s(B_i) \ | \ 1\leq i\leq t; s\in S\}$$
is a $(G\times H,G\times\{0\}, \, ^n(2K),2\lambda)$-DF.
Now note that we have:
$$\biguplus_{s\in S}\widehat\varepsilon_s(B_i)=\biguplus_{s\in S}\biguplus_{d\in D_i}\{d\}\times\{sf(d),-sf(d)\}=$$
$$\biguplus_{d\in D_i}\{d\}\times\biguplus_{s\in S}\{s,-s\}\cdot f(d)=\biguplus_{d\in D_i}\{d\}\times(H\setminus\{0\})=D_i\times(H\setminus\{0\}).$$
Recalling that the $D_i$'s partition $G$ since $\cal D$ is partitioned, we conclude that the blocks of $\cal F$ partition  
$G\times(H\setminus\{0\})=(G\times H)\setminus(G\times\{0\})$. Hence $\cal F$ is a PDF relative to $G\times\{0\}$.

In order to ``complete" $\cal F$ to an ordinary PDF in $G\times H$ we need a set of blocks partitioning
$G\times\{0\}$ whose list of differences gives $2\lambda$ times $(G\setminus\{0\})\times \{0\}$. Such a set is trivially given either by the singleton
$\{G\times\{0\}\}$ or by $\{D_i\times\{0\} \ | \ 1\leq i\leq t\}$.
 We conclude that $$\{\widehat\varepsilon_s(B_i) \ | \ 1\leq i\leq t; s\in S\} \ \cup \ \{G\times\{0\}\}$$
 is a $(2\lambda(2n+1), \,^n(2K) \ \uplus \ \{2\lambda\}, \ 2\lambda)$-PDF in $G\times H$ and that 
 $$\{\widehat\varepsilon_s(B_i) \ | \ 1\leq i\leq t; s\in S\} \ \cup \ \{D_i\times\{0\} \ | \ 1\leq i\leq t\}$$
  is a $(2\lambda(2n+1),\,^n(2K) \ \uplus \ K, \ 2\lambda)$-PDF in $G\times H$.\end{proof}

As a corollary we get an infinite class of PDFs applying the above theorem with the use of the
Hadamard PDFs of Proposition \ref{D,G-D}.

\begin{corollary}\label{cor1}
If $D$ is $(4u^2,2u^2-u,u^2-u)$-HDS in $G$ and the maximal prime power divisors of $2n+1$ are all greater than $4u^2+2u$, then
there exists a $$(4u^2(2n+1),\ [^n(4u^2-2u),\,^1(4u^2),\,^n(4u^2+2u)], \ 4u^2]\mbox{-PDF}$$
and a 
$$(4u^2(2n+1),\ [^n(4u^2-2u),\,^1(2u^2-u),\, ^1(2u^2+u), \ ^n(4u^2+2u)], \ 4u^2]\mbox{-PDF}.$$
\end{corollary}
\begin{proof}
Let $q_1,\dots, q_t$ be the maximal prime power divisors of $2n+1$, assume that each $q_i$ is greater than $4u^2+2u$,
and let $\rho_i$ be a primitive root of $ \mathbb{F}_{q_i}$. Then the assertion will follow by applying Theorem \ref{HPDF} with $\cal D$ the Hadamard PDF 
of Proposition \ref{D,G-D}, with $R$ the ring $\mathbb{F}_{q_1}\times\dots\times \mathbb{F}_{q_t}$, and with $Y=\{(\rho_1^j,\dots,\rho_t^j) \ | \ 1\leq j\leq2u^2+u\}$.
\end{proof}

Applying the above corollary with $u=1$, namely using the trivial $(4,1,0)$-HDS, one obtains a cyclic
$(8n+4, [^n2,\,^14,\,^n6], \ 4)$-PDF whenever $2n+1$ is coprime with 15. 

It is known that there exists a $(16,6,2)$-HDS in every group $G$ of order 16
except $G=\Z_{16}$. Then Corollary \ref{cor1} gives a
$(400, \ [^{12}12, \,^116, \,^{12}20], 16)$-PDF in $G\times \Z_5\times\Z_5$ for any $G$ as above. Note, however, that 
Theorem \ref{HPDF} cannot produce a PDF with the same parameters in $G\times\Z_{25}$ 
since a set $Y$ of units of $\Z_{25}$ such that $\Delta(Y \ \cup \ -Y)\subset U(\Z_{25})$ has size at most 3.

We get another {\it sporadic} class of PDFs applying our main theorem with the use of the
$(32,[^22,6,22],16)$-PDF  given at the end of the introduction.

\begin{corollary}
If the maximal prime power divisors of $2n+1$ are all greater than $44$, then
there exists a $$(64n+32,\ [\,^{2n}4,\, ^n12, \,^132,\,^n44], \ 32)\mbox{-PDF}$$ and a $$(64n+32,\ [^22,\,^{2n}4,\, ^16, \,^n 12, \,^122,\,^n44], \ 32)\mbox{-PDF}.$$
\end{corollary}

\medskip
The first value of $n$ for which the above corollary can be applied is 23. In this way one gets a $(1504,[^{46}4,\,^{23}12,\,^132,\,^{23}44],32)$-PDF.

\section*{Acknowledgement}
This work has been performed under the auspices of the G.N.S.A.G.A. of the
C.N.R. (National Research Council) of Italy.

\end{document}